\documentclass[12pt,oneside]{amsart}
\usepackage{amscd,amsmath,amssymb,amsfonts,a4}
\usepackage{hyperref}
\usepackage{color}

\newlength{\rulebreite}

\def\timesover#1#2#3{\ \xymatrix@1@=0pt@M=0pt{ _{#1}&\times&_{#2} \\& ^{#3}&}\ }
\def\otimesover#1#2#3{\ \xymatrix@1@=0pt@M=0pt{ _{#1}&\otimes&_{#2} \\& ^{#3}&}\ }

\usepackage[all]{xy}
\theoremstyle{plain}
\newtheorem{thm}{Theorem}
\newtheorem{lem}[thm]{Lemma}
\newtheorem{cor}[thm]{Corollary}
\newtheorem{prop}[thm]{Proposition}
\theoremstyle{definition}

\newtheorem{defn}[thm]{Definition}

\newtheorem{claim}[thm]{Claim}
\numberwithin{thm}{section}
\numberwithin{equation}{section}

\newcommand{\ml}[2]{\begin{multline}\label{#1}#2 \end{multline}}
\newcommand{\ga}[2]{
\begin{gather}\label{#1}#2\end{gather} 
}

\newcommand{\sC}{{\mathcal C}}
\newcommand{\sD}{{\mathcal D}}
\newcommand{\sE}{{\mathcal E}}

\newcommand{\sH}{{\mathcal H}}
\newcommand{\sI}{{\mathcal I}}

\newcommand{\sK}{{\mathcal K}}

\newcommand{\sO}{{\mathcal O}}

\newcommand{\sQ}{{\mathcal Q}}

\newcommand{\C}{{\mathbb C}}

\renewcommand{\H}{{\mathbb H}}

\renewcommand{\P}{{\mathbb P}}
\newcommand{\Q}{{\mathbb Q}}
\newcommand{\R}{{\mathbb R}}

\newcommand{\Z}{{\mathbb Z}}

\begin{document}

\title[]{Algebraic differential characters of flat connections with nilpotent residues}
\author{H\'el\`ene Esnault}
\address{
Universit\"at Duisburg-Essen, Mathematik, 45117 Essen, Germany}
\email{esnault@uni-due.de}
\date{October 20, 2007}
\thanks{Partially supported by  the DFG Leibniz Prize}
\begin{abstract}
We construct unramified algebraic differential characters for flat connections with nilpotent residues along a strict normal crossings divisor. 
\end{abstract}
\maketitle
\section{Introduction}
In \cite{ChS}, Chern and Simons defined  classes $ \hat{c}_n((E,\nabla))\in H^{2n-1}(X, \R/\Z(n))$ for $n\ge 1$ and 
 a flat bundle $(E,\nabla)$ on a $\sC^\infty$ manifold $X$, where $\Z(n):=\Z\cdot (2\pi \sqrt{-1})^n$. Cheeger and Simons defined in \cite{CS} the group of real $\sC^\infty$ differential characters $ \hat{H}^{2n-1}(X, \R/\Z)$, which is an extension of global $\R$-valued $2n$-closed forms with $\Z(n)$-periods  by $
H^{2n-1}(X, \R/\Z(n))$. They show that  the Chern-Simons classes extend to classes $\hat{c}_n((E,\nabla)) \in
\hat{H}^{2n-1}(X, \R/\Z)
$, if $\nabla$ is a (not necessarily flat) connection, such that the associated differential form is the Chern form computing the $n$-th Chern class associated to the curvature of $\nabla$.

If $X$ now is a complex manifold, and $(E,\nabla)$ is a bundle with an algebraic connection, Chern-Simons and Cheeger-Simons invariants give classes  
$\hat{c}_n((E,\nabla))\in \hat{H}^{2n-1}(X_{{\rm an}}, \C/\Z)$ with a similar definition of complex $\sC^\infty$ differential characters. Those classes have been studied by various authors, and most remarquably, it was shown by A. Reznikov 
that if $X$ is projective and $(E,\nabla)$ is flat, then the classes ${\hat c}_n((E,\nabla))$ are torsion, for $n\ge 2$. This answered positively a conjecture by S. Bloch \cite{B}, which echoed a similar conjecture by Cheeger-Simons in the $\sC^\infty$ category \cite{Ch},  \cite{CS}. 

On the other hand, for $X$ a smooth  complex algebraic variety, we defined in \cite{ADC}
the group $AD^n(X)$  of algebraic differential characters. It is easily written as the hypercohomology group $\H^n(X, \sK_n\xrightarrow{d\log} \Omega^n_X\xrightarrow{d} \Omega^{n+1}_X\to \ldots \xrightarrow{d} \Omega^{2n-1}_X)$, where $\sK_n$ is the Zariski sheaf of Milnor $K$-theory which is unramified in codimension 1.  It has the property that it maps to the Chow group $CH^n(X)$, to algebraic closed $2n$-forms
which have $\Z(n)$-periods, and to the complex $\sC^\infty$ differential characters
  $\hat{H}^{2n-1}(X_{{\rm an}}, \C/\Z)$. If $(E,\nabla)$ is a bundle with an algebraic connection, it has classes $c_n((E,\nabla))\in AD^n(X)$ which lift both the Chern classes of $E$ in $CH^n(X)$ and $\hat{c}_n((E,\nabla))$. All those constructions are contravariant in $(X, (E,\nabla))$, the differential characters have an algebra structure, and the classes fulfill the Whitney product formula. They admit a logarithmic version: if $j: U\to  X$ is a (partial) smooth compactification of $U$ such that $D:= X\setminus U$ is a strict normal crossings divisor, one defines the group $AD^n(X, D)=\H^n( X, \sK_n \xrightarrow{d\log} \Omega^n_{ X}(\log D) \xrightarrow{d} \Omega^{n+1}_{ X}(\log D)\to \ldots \xrightarrow{d} \Omega^{2n-1}_{ X}(\log D))$. Obviously one has maps $AD^n( X)\to AD^i( X, D)\to AD^n(U)$. 
The point is that if $(E,\nabla)$
extends a pole free  connection $(E,\nabla)|_U$ to a connection on $X$ with logarithmic poles along $D$, then $c_n((E,\nabla)|_U)\in AD^n(U) $ lifts to well defined classes  $ c_n( (E,  \nabla))\in AD^n( X, D)$ with the same functoriality and additivity properties. 

If $X$ is a smooth algebraic variety defined over a characteristic 0 field, and $X\supset U$ is a smooth (partial) compactification of $U$,  it is computed in \cite[Appendix~B]{EV} that one can express the Atiyah class (\cite{A}) of a bundle  extension $E$ of $E|_U$ in terms the residues of the extension $\nabla$ of $\nabla|_U$ along $D= X\setminus U$. In particular, if $ X$ is projective, $ \nabla$ has logarithmic poles along $D$  and has  nilpotent residues, one obtains that the de Rham Chern classes of $ E$ are zero. If $k=\C$, this implies that the (analytic) Chern classes of $ E$ in Deligne-Beilinson cohomology $H^{2n}_{\sD}( X,\Z(n))$ lie  in the continuous part  $ H^{2n-1}( X_{{\rm an}}, \C/\Z(n))/F^n \subset  H^{2n}_{\sD}( X, \Z(n))$. 

The purpose of this note is to show that this lifting property is in fact stronger:
\begin{thm} \label{thm1}
Let $ X\supset U$ be a smooth (partial) compactification of a  complex variety $U$, such that $D=\sum_j D_j= X\setminus U$ is a strict normal crossings divisor. Let $(E, \nabla)$ be a flat connection with logarithmic poles along  $D$ such that its residues $\Gamma_j$ along $D_j$ are all nilpotent.  Then the classes $c_n(( E,  \nabla)) \in AD^n( X,  D)$ lift to well defined classes
$c_n(( E,  \nabla, \Gamma)) \in AD^n( X)$, which  satisfy the Whitney product formula. More precisely, the classes $c_n(( E,  \nabla, \Gamma))$ lie  in the subgroup $AD^n_{\infty}( X)=\H^n( X, \sK_n\xrightarrow{d\log} \Omega^n_{\bar X}\xrightarrow{d} \Omega^{n+1}_{ X}\to \ldots \xrightarrow{d} \Omega^{{\rm dim}(X)}_{ X})\subset AD^n( X) $ of classes mapping to 0 in $H^{0}( X, \Omega^{ 2n}_X)$. 
\end{thm}
They also fulfill some functoriality property, and one can express what their restriction to the various strata of $D$ precisely are.
 
Let us denote by $\hat{c}_n(( E,  \nabla, \Gamma))$ the image of 
$c_n(( E, \nabla, \Gamma))$ via the regulator map $AD^n( X)\to \hat{H}^{2n-1}( X_{{\rm an}}, \C/\Z)$ defined in \cite{E2} and \cite{ADC}, which restricts to a regulator map $AD^n_{\infty}( X) \to H^{2n-1}( X_{{\rm an}}, \C/\Z(n))$.  
As an immediate consequence, one obtains
\begin{cor}
 Let $(X, ( E, \nabla, \Gamma))$ be as in the theorem. Then the Cheeger-Chern-Simons classes 
$\hat{c}_n((E,\nabla)|_U)\in H^{2n-1}(U_{{\rm an}}, \C/\Z(n))\subset \hat{H}^{2n-1}(U_{{\rm an}}, \C/\Z)$ lift to well defined classes
$\hat{c}_n(( E, \nabla, \Gamma)) \in H^{2n-1}( X_{{\rm an}}, \C/\Z(n)) \subset \hat{H}^{2n-1}( X_{{\rm an}}, \C/\Z)$, with the same properties.
\end{cor}
A direct $\sC^\infty$ construction of 
$\hat{c}_n(( E,  \nabla, \Gamma))\in H^{2n-1}(X_{{\rm an}}, \C/\Z(n))$ in the spirit of Cheeger-Chern-Simons has been performed by P. Deligne
and is written in a letter of P. Deligne to the authors of \cite{IS}.
It consists in modifying the given connection $ \nabla$ by a $\sC^\infty$ one form with values in $\sE nd( E)$, so as to obtain a (possibly non-flat) connection without residues along $D$. This modified connection admits classes in $H^{2n-1}(X_{{\rm an}}, \C/\Z(n))\subset \hat{H}^{2n-1}( X_{{\rm an}}, \C/\Z)$. 
That they do not depend on the choice of the one form  relies essentially on the  argument  showing that if $\nabla$ is flat with logarithmic poles along $D$ (and without further conditions on the residues), for $n\ge 2$, the image of $c_n((E,\nabla))$ in $H^0(U, \sH^{2n-1}_{DR})$, where $\sH^j_{DR}$ is the Zariksi sheaf of $j$-th de Rham cohomology,  in fact lies in the unramified cohomology $H^0(X,\sH^{2n-1}_{DR})\subset H^0(U, \sH^{2n-1}_{DR})$. For this, see \cite[Theorem~6.1.1]{BE}.
In the case when $D$ is smooth, J. Iyer and C. Simpson constructed the $\sC^\infty$ classes 
$\hat{c}_n(( E,  \nabla, \Gamma))\in H^{2n-1}(X_{{\rm an}}, \C/\Z(n))$ using the existence of the  
$\sC^\infty$ trivialization of the canonical extension after an \'etale cover, a fact written by Deligne in a letter, together with Deligne's suggestion of considering patched connection. They then show that Reznikov's argument and theorem \cite{R} adapts to those classes.  
Our note is motivated by the question raised in \cite{IS}  on the construction in the general case.

Our algebraic construction in theorem \ref{thm1}   relies on the modified splitting 
principle developed in \cite{E}, \cite{E2} and \cite{ADC} in order to define the classes in $AD^n(X, D)$. Let $q: Q\to  X$ be the complete flag bundle of $ E$. A flat connection on $ E$ 
with logarithmic poles along $D$ defines a map of differential graded algebras  $\tau:\Omega^\bullet_{Q}(\log q^{-1}(D)) \to \sK^\bullet$ where 
$\sK^i=q^*\Omega^i_{ X}(\log D)$ and $\ Rq_*\sK^\bullet=\Omega^\bullet_{X}(\log D)$. 
This defines a partial flat connection $\tau \circ q^*\nabla: q^*  E\to q^*\Omega^1_{ X}(\log D) \otimes_{\sO_{\sQ}} q^* E$ which has the property that it stabilizes all the rank one subquotients of $q^* E$. 
On the other hand, the nilpotency of $\Gamma$ allows to filter the restriction  $E|_{\Sigma}$ to the different strata $\Sigma$ of $D$, in such a way that the restriction $\nabla|_\Sigma: E|_\Sigma
\to \Omega^1_X(\log D)|_\Sigma\otimes E|_{\Sigma}$  of the connection stabilizes the filtration $F^\bullet_{\Sigma}$, and has the following important extra property: the induced flat connection $\nabla|_{\Sigma}$  on $gr(F^\bullet_\Sigma)$ has values in $\Omega^1_{\Sigma}(\log {\rm rest})$, where rest is the interestion with $\Sigma$ of the part of $D$ which is
transversal to $\Sigma$. This fact translates into a sort of stratification of the flag bundle $Q$, where $\tau$ is refined on this stratification and has values in the pull back of $\Omega^1_\Sigma({\rm rest})$. 
Modulo some  geometry in $Q$, the next observation consists 
 in expressing  the sections  $\alpha \in \Omega^i_{ X}$ of forms without poles 
as pairs $\alpha=(\beta \oplus \gamma) \in \Omega^1_{ X}(\log D)\oplus \Omega^1_D$ such that $\beta|_D=\gamma$, where $\Omega^i_D=\Omega^i_{ X}/\Omega^i_{ X}(\log D)(-D)\subset \Omega^i_X(\log D)|_D$. This yields a complex  receiving quasi-isomorphically $\Omega^{\ge i}_X$, which is convenient to define the wished classes. 
\\ \ \\
{\it Acknowledgement:} Our algebraic construction was performed independently of P. Deligne's  $\sC^\infty$ construction sketched above. 
We thank C. Simpson for sending us  afterwards Deligne's letter. We also thank him for pointing out a mistake in an earlier version of this note. We thank E. Viehweg for his encouragement and for  discussions on the subject, which reminded us of the discussions we had when we wrote \cite[Appendix~C]{EV}.

\section{Filtrations}

Let $X$ be a smooth variety defined over a characteristic 0 field $k$. Let $D\subset X$ be a strict normal crossings divisor (i.e. the irreducible components are smooth over $k$), and let $(E,\nabla)$ be a  connection $\nabla: E\to \Omega_X^1(\log D)\otimes E$ with residue $\Gamma$ defined by the composition
\ga{2.1}{\xymatrix{\ar[drr]_{\Gamma} E \ar[rr]^<<<<<<<<<{\nabla} & & \Omega^1_X(\log D) \otimes_{\sO_X} E \ar[d]^{1\otimes {\rm res}} \\
& &  \nu_*  \sO_{D^{(1)}}  \otimes_{\sO_X} E}}
where $D^{(1)}=\sqcup_j D_j$.
The composition of $\Gamma$ with the projection $\nu_*\sO_{D^{(1)}}\to \sO_{D_j}$ 
defines $\Gamma_j: E\to \sO_{D_j}\otimes E$ which factors through $\Gamma_j \in {\rm End}(
\sO_{D_j}\otimes E)$.  
We write 
\ga{2.2}{\Gamma \in {\rm Hom}_{\sO_X}(\sO_D \otimes_{\sO_X} E ,
\nu_*  \sO_{D^{(1)}}  \otimes_{\sO_X} E).}
Recall that if $\nabla$ is integrable, then
\ga{2.3}{[\Gamma_i|_{D_{ij}}, \Gamma_j|_{D_{ij}}]=0.}
We use the notation $D_I=D_{i_1}\cap \ldots \cap D_{i_r}$ if $I=\{i_1,\ldots, i_r\}$, $D=D^I+\sum_{s\in I} D_s $ with $D^I=\sum_{\ell \notin I} D_\ell$.
The connection $\nabla: E\to \Omega^1_X(\log D)\otimes E$  stabilizes $E(-D_j)$, but also $E\otimes \sI_{D_{I}}$, as the K\"ahler differential on $\sO_X$ restricts to a flat $\Omega^1_X(\log(\sum_{s\in I} D_s))$-connection on $ \sI_{D_I}$. Thus $\nabla$ induces a flat connection
\ga{2.4}{\nabla_I: E|_{D_{I}} \to \Omega^1_X(\log D)|_{D_{I}} \otimes E|_{D_{I}}.}
One has the diagram 
\ga{2.5}{\xymatrix{  & \ar[d] \ar[ddr]_>>>>>>>>>{{\rm 1\otimes \Gamma_j}}  \Omega^1_{D_j}(\log (D^j\cap D_j))\otimes E  & \ar[d] \Omega^2_{D_j}(\log (D^j\cap D_j)) \otimes E\\
\ar[dr]_{\Gamma_j} E|_{D_j} \ar[r]^<<<<{\nabla_j} & \ar[d]_{{\rm res}} \Omega^1_X(\log D)|_{D_j} \otimes E \ar[r]^{\nabla_j} & \ar[d]_{{\rm res}} \Omega^2_X(\log D)|_{D_j}\otimes E\\
 &  E|_{D_j}  & \Omega^1_{D_j}(\log (D^j\cap D_j))\otimes E 
}
}
We define $F_j^1={\rm Ker}(\Gamma_j)\subset E|_{D_j}$. It is a coherent subsheaf. $\nabla_j$ sends $F_j^1$ to $\Omega^1_{D_j}(\log D^j\cap D_j)\otimes E$, but because of integrability, the diagram \eqref{2.3} shows that $\nabla_{D_j}$ induces a flat connection $F_j^1\to \Omega^1_{D_j}(\log D^j\cap D_j)\otimes F_j^1$.
\begin{claim} \label{claim2:1}
 $F_j^1\subset E|_{D_j}$ is a subbundle. 
\end{claim}
\begin{proof}
We use  Deligne's Riemann-Hilbert correspondence \cite{DPSR}:
the data are defined over a field of finite type $k_0$ over $\Q$, so embeddable in $\C$, and the question is compatibe with the base changes $\otimes_{k_0}k$ and $\otimes_k \C$. So it is enough to consider the question for the underlying analytic connection on a polydisk $(\Delta^*)^r\times \Delta^s$ with coordinates $x_j$, where $D_j$ is defined by $x_j=0$ for $1\le j\le r$. By the Riemann-Hilbert correspondence, the argument given in \cite[p.86]{DPSR}
shows that the analytic connection is isomorphic  to 
$(V\otimes \sO, \sum_1^r \Gamma_j^0  \frac{dx_i}{x_i})$, where the matrices $\Gamma_j^0$ are constant nilpotent. Thus   $F_j^1$ is isomorphic to 
 $F_j^1(V) \otimes \sO_{D_j}$ on the polydisk, with $F^1_j(V):={\rm Ker}(\Gamma_j^0)$,  thus is a subbundle. 
\end{proof}
We can replace $E|_{D_j}$ by $E|_{D_j}/F_j^1$ in \eqref{2.4} and redo the construction. This defines by pull back $ F^2_j\subset E|_{D_j}\twoheadrightarrow {\rm Ker}(\Gamma_j:  E|_{D_j}/F^1_j \to E|_{D_j}/F^1_j)$ with $F^2_j\supset F_j^1$ etc. 
\begin{claim} \label{claim2:2}
$F^\bullet_j: F^0_j=0\subset F^1_j\subset \ldots \subset F_j^i \subset \ldots \subset F^{r_j}_j= E|_{D_j}$ is a filtration by subbundles with a flat $\Omega^1_{X}(\log D)|_{D_j}$-valued connection,  such that the induced connection $\nabla_j$ on $gr (F_j^\bullet)$ is  flat and 
$\Omega^1_{D_j}(\log D^j\cap D_j)$-valued. (One can also tautologically say that $F^\bullet_j$ refines the (trivial) filtration on $E|_{D_j}$). 
\end{claim}
\begin{proof}
By construction, the flat $\Omega^1_X(\log D)|_{D_j}$-valued connection $\nabla_j$ on 
$E|_{D_j}$ respects the filtration and induces a flat $\Omega^1_{D_j}(\log D^j\cap D_j)$-connection on $gr (F_j^\bullet)$. 
 We use the transcendental argument to show that this is a filtration by subbundles. With the notations as in the proof of the claim \ref{claim2:1}, $F^s_j$ is analytically isomorphic to $F_j^s(V)\otimes \sO_{D_j}$, where $F^1_j(V)\subset F^2_j(V)\subset \ldots \subset V$ is the filtration on $V$ defined by the successive kernels of $\Gamma_j^0$, so $F^2_j(V)$ is the inverse image of ${\rm Ker}(\Gamma_j^0)$ on  $V/F^1_j(V)$ etc. 
\end{proof}
The  argument which allows us to construct $F_j^\bullet$  can in  be used to define 
successive refinements on all $E|_{D_I}$. 
We consider now the case  $|I|=r\ge 2$. We  refine the filtrations $F_J^\bullet|_{D_I}$, which have been constructed inductively, where $J\subset I, |J|<r$ . 
In fact, we do the construction directly on $E|_{D_I}$.
We have $r$ linear maps induced by $\Gamma_j$
\ga{2.6}{\Gamma_j|_{D_{I}}: E|_{D_{I}} \xrightarrow{\nabla_{I}} 
\Omega^1_X(\log( D))|_{D_{I}}\otimes E|_{D_{I}} \to  \sO_{D_j}\otimes E|_{D_{I}}=E_{D_{I}}}
 We define 
\ga{2.7}{F_{I}^1=\cap_{j\in I} {\rm Ker}(\Gamma_j|_{D_{I}})
=\cap_{j\in I}F^1_j|_{D_{I}}.}
\begin{claim} \label{claim2:3}
 $F_{I}^1 \subset E|_{D_{I}}$ is a subbundle, stabilized by the connection $\nabla_{I}$, and more precisely one has $\nabla_{I}: F^1_{I}\to \Omega^1_{D_{I}}(\log (D^{I}\cap D_{I}))\otimes F^1_{I}.$
\end{claim}
\begin{proof}
We argue analytically as in the proof of claim \ref{claim2:1}. 
 With  notations as there, 
 the analytic  $F^1_{I}$ isomorphic to $F^1_{I}(V)\otimes \sO_{D_{I}}$.
\end{proof}
Thus $\nabla_{I}$ induces a flat $\Omega^1_X(\log D)|_{D_{I}}$-valued connection on the quotient $E|_{D_{I}}/F^1_{I}$. 
We define $ F^2_{I}\supset F^1_{I}$ in $E|_{D_{I}}$ to be the inverse image via the projection $E|_{D_{I}}\to E|_{D_{I}}/F^1_{I}$ of $\cap_{j\in I}{\rm Ker}(\Gamma_j|_{D_{I}})
 $, etc. 
\begin{claim} \label{claim2:4}
 The filtration $F^\bullet_{I}: F^0_{I}=0\subset F^1_{I}\subset F^2_{I}\subset \ldots \subset F^{r_{I}}_{I}=E|_{D_{I}}$ is a filtration by subbundles, stabilized by $\nabla_{I}$, such that $\nabla_{I}$ on $gr(F^\bullet_{I})$ is a flat $\Omega^1_{D_{I}}(\log (D^{I}\cap D_{I}))$-valued connection. Furthermore, $F^\bullet_{I}$ refines all $F^\bullet_J|_{D_{I}}$ for all $J\subset I, |J|<r$ and one has  compatibility of the refinements in the sense that if $K\subset J\subset I$, then the refinement $F_I^\bullet$ of $F_K^\bullet|_{D_I}$ is 
the composition of the refinements $F^\bullet_I$ of $F_J^\bullet|_{D_J}$ and $F_J^\bullet$ of $F_K^\bullet|_{D_J}$. 
\end{claim}
\begin{proof}
 We argue again analytically. Then  $F^s_{I}$ is isomorphic to $F^s_{I}(V)\otimes \sO_{D_I}$ with the same definition. The filtration terminates as finitely many mutually commuting nilpotent endomorphisms on a finite dimensional vector space always have a common eigenvector.

\end{proof}
\begin{defn} \label{defn2:5}
We call $F_I^\bullet$  the canonical filtration of $E|_{D_I}$ associated to $\nabla$, which defines $(gr(F_I^\bullet), \nabla_I, \Gamma_I)$ where $\nabla_I$ is the flat $\Omega^1_I(\log (D^I\cap D_I)$-valued connection on $gr(F_I^\bullet)$, and $\Gamma_I$ is its nilpotent residue along the normalization of $D^I\cap D_I$.  

\end{defn}
 \begin{proof}
 
 \end{proof}
\section{$\tau$-Splittings}
We first define flag bundles. We set $q_I: Q_I\to D_I$ to be the total  flag bundle associated to $E|_{D_I}$. So the pull back of $E|_{D_I}$ to $Q_I$ has a filtration by subbundles such that the associated graded bundle is a sum of rank one bundles $\xi_I^s$ for $s=1,\ldots, N={\rm rank}(E)$. (It is here understood that $D_{\emptyset}=X$, and to simplify, we set $q=q_{\emptyset}: Q\to X, Q_{\emptyset}=Q$). For $J\subset I$, the inclusion $D_I\to D_I$ defines inclusions $i(J\subset I): Q_I\to Q_J$.
The canonical filtrations associated to $\nabla$  allow to define partial sections of the $q_I$. 
As an illustration, let us assume that $I=\{1\}$, thus $D$ is smooth,  and
that  $F^\bullet_{1}$ is a total flag, i.e. the $gr(F_1^\bullet)$ is a sum of rank one bundles. Then $F_1^\bullet$ defines a section $D\xrightarrow{\lambda^F_1} Q$. 

More generally, let us define $G_I^s=F^s_I/F_I^{s-1}$. We define 
\ga{3.1}{\xymatrix{ \ar[d]_{q_I} Q_I & \ar[l]_{\lambda_I^F} Q_I^F \ar[dl]^{q_I^F}\\
D_I}}
using the filtration: recall that $Q_I\to D_I$ is the composition of $\P(E|_{D_I})\to D_I$ with $\P(E')\to \P(E|_{D_I})$ etc., where $E'\to \sO_{\P(E)}\otimes E$ is the rank $(N-1)$ subbundle defined as the kernel to the rank one canonical rank 1 bundle $\xi_I^N(\P(E|_{D_I}))$, the pull back of which to $Q_I$ defines the last graded rank one quotient. Then the quotient  $E|_{D_I}\to G^{r_I}_I$ defines a map $\P(E|_{D_I})\xleftarrow{} \P(G_I^{r_I})$
such that the pull back of $\xi_I^N(\P(E|_{D_I}))$ is $\xi$, where $\xi$ is the canonical rank one bundle. Writing $G'\to G^{r_I}_I$ for the kernel, we redo the same construction for 
$E', G'$ replacing $E|_{D_I}, G^{r_I}_I$ etc. We find this way that the flag bundle of $G^{r_I}_I$ maps to the intermediate step between $D_I$ and $Q_I$
which splits the first $M$ rank one bundles, where $M$ is the rank of $G^{r_I}_I$. Then we continue with the pull back of $G^{r_I-1}_I$  to the flag bundle of $G^{r_I}_I$, replacing $G^{r_I}_I$, and $E''$ replacing $E$, where $E''$ on this intermediate step is the rank $N-M$ bundle which is not yet split. All this is very classical.

We have extra closed embeddings $\lambda^F(I\subset J)$ which come from the refinements 
of the canonical filtrations, which are described in the same way: for $J\subset I$, one has commutative squares
\ga{3.2}{\xymatrix{\ar[d]_{q_J^F} Q_J^F & & \ar[ll]_{\lambda^F(I\subset J)} Q_I^F \ar[d]^{q_I^F}\\
D_J & & \ar[ll]^{i(I\subset J)} D_I
} 
\xymatrix{\ar[d]_{q} Q & & \ar[ll]_{\mu_I} Q_I^F \ar[d]^{q_I^F}\\
X & & \ar[ll]^{i_I} D_I
} 
}
where $i_I=i(\emptyset \subset I), \ \mu_I=\lambda(\emptyset\subset I)$. 

Recall from  \cite{E}, \cite{E2}, \cite{ADC} that $\nabla$ yields a splitting 
$\tau: \Omega^1_Q(\log q^{-1}(D))\to q^*\Omega^1_X(\log D)$, and that flatness of $\nabla$ implies flatness of $\tau$ in the sense that it induces a map of differential graded algebras $(\Omega^\bullet_Q(\log  q^{-1}(D)),d)\to (q^*\Omega^\bullet_X(\log D), d_\tau)$  so in particular, $(Rq_*\Omega^{\ge n}_X(\log D),d)=(\Omega^{\ge n}_X(\log D), d)$. Furthermore, the filtration on $q^*(E)$ which defines the rank one subquotient $\xi^s$ has the property that it is stabilized by $\tau\circ q^*\nabla$, and this defines a  $\tau$-flat connection $\xi^s\to q^*\Omega^1_X(\log D)\otimes \xi^s$. 

The $\tau$-splitting is constructed first on $\P(E)$, with $p: \P(E)\to X$. Then $\tau\circ \nabla$ stabilizes the beginning of the flag  $E'\subset $pull-back of $E$ etc. Concretely, the composition $\Omega^1_{\P(E)/X}(1)\xrightarrow{\nabla} \Omega^1_{\P(E)}\otimes E \xrightarrow{{\rm projection}} \Omega^1_{\P(E)}\otimes \sO_{\P(E)}(1)$ defines the splitting. 
On the other hand, the flat $\Omega^1_X(\log D)|_{D_I}$-valued connection on $G^{r_I}_I$ has values in $\Omega^1_{D_I}(\log (D_I\cap D^I))$.

When we restrict to $\P(G^{r_I}_I)$, then one has a factorization
\ga{3.3}{\xymatrix{\ar[dr]_{\tau} \Omega^1_{\P(E)}(\log p^{-1}(D)) \otimes \sO_{\P(G^{r_I}_I)} \ar[r]^{\tau(G^{r_I}_I)} & \Omega^1_{D_I}(\log (D_I\cap D^I)) \otimes \sO_{\P(G^{r_I}_I)} \ar[d]^{{\rm inj}} \\
& \Omega^1_X(\log D) \otimes  \sO_{\P(G^{r_I}_I)}
}
}
which defines a differential graded algebra $ (\Omega^\bullet_{D_I}(\log (D_I\cap D^I)) \otimes \sO_{\P(G^{r_I}_I)}, d_\tau)$ with total direct image on $D_I$ being $(\Omega^\bullet_{D_I}(\log (D_I\cap D^I)),d) $ and with the property that $\xi$ has a flat connection with values in $\Omega^1_{D_I}(\log (D_I\cap D^I))$, which is compatible with the flat $p^*\Omega^1_X(\log D)$-connection on  $\xi^N$. We can repeat the construction with $D_I\to X$ replaced by $\P(G^{r_I}_I)\to \P(E|_{D_I})$, with $E|_{D_I} \to G^{r_I}_I$ replaced by 
$E'\to G'$ where $E'={\rm Ker}( E|_{D_I}\otimes \sO_{\P(E|_{D_I})}\to \sO(1))$ and 
$G'={\rm Ker}(G^{r_I}_I \to \sO(1)$. This splits the next rank one piece, one still has the splitting as in \eqref{3.3}, and we go on till we reach the total flag bundle to $G^{r_I}_I$. Then we continue with the flag bundle to $G^{r_I-1}_I$ etc. 
We conclude
\begin{claim} \label{claim3:1} One has a factorization 
\ga{3.4}{\xymatrix{\ar[dr]_{\tau} \mu_I^*\Omega^1_{Q}(\log q^{-1}(D))  \ar[r]^{\tau_I} & (q_I^F)^*\Omega^1_{D_I}(\log (D_I\cap D^I)) \ar[d]^{{\rm inj}}  \\
& (q_I^F)^*\Omega^1_X(\log D)|_{D_I} 
}} 
$\tau_I$ defines a differential graded algebra $((q_I^F)^*\Omega^\bullet_{D_I}(\log (D_I\cap D^I)), d_\tau)$ which is a quotient of $\mu_I^*(\Omega^\bullet_{Q}(\log q^{-1}(D)), d)$.
The flat $q^*\Omega^1_X(\log D)$-valued $\tau$-connection on $\xi^s, s=1,\ldots, N$, restricts via 
the splitting $\tau_I$, to a flat $(q_I^F)^*\Omega^1_{D_I}(\log (D^I\cap D_I))$-valued $\tau$-connection on $(\xi^F_I)^s= \mu_I^*\xi^s$.  
\end{claim}
\begin{defn} \label{defn3:2}
On $Q$ we define the complex of sheaves $$A(n)=A^n\to A^{n+1}\to \ldots$$ with
\ml{}{A^i= B^i\oplus C^i\\  B^i=\oplus_I (\mu_I)_* (q_I^F)^*\Omega^i_{D_I}(\log (D^I\cap D_I)), \ C^i=\oplus_{I\neq \emptyset}  (\mu_I)_*(q_I^F)^*\Omega^{i-1}_X(\log D)|_{D_I},\notag } where $C^i=0$ for $i=n$.  The differentials $D_\tau$  are defined as follows: $(\oplus_I\beta_I, \oplus_I \gamma_I)$, where $\beta_I \in (\mu_I)_* (q_I^F)^*\Omega^i_{D_I}(\log (D^I\cap D_I)), \gamma_I\in 
(\mu_I)_*(q_I^F)^*\Omega^{i-1}_X(\log D)|_{D_I}$
 is sent to 
\ml{}{\oplus_I d_\tau \beta_I \in (\mu_I)_* (q_I^F)^*\Omega^{i+1}_{D_I}(\log (D^I\cap D_I)),\\ \oplus_I d_\tau \gamma_I +(-1)^i (\mu_I^*\beta-\beta_I) \in (\mu_I)_*(q_I^F)^*\Omega^{i}_X(\log D)|_{D_I}.\notag} 

\end{defn}
Let $\sK_n$ be the image of the Zariski sheaf of Milnor $K$-theory into Milnor $K$-theory $K_n(k(X)$ of the function field (which is the same as Ker($K_n(k(X))\to \oplus K_{n-1}(\kappa(x)))$ on all codimension 1 points $x\in X$). The $\tau$-differential defines $ d_\tau \log: \sK_n\to A^n=B^n$ ($C^n=0$). The image in $A^n$ is $D_\tau$-flat. Thus this defines  $ d_\tau\log: \sK_n\to A(n)[-1]$.

\begin{defn} \label{defn3:3}
We define $\sK_n\Omega^\infty_Q$ to be the complex $ \sK_n\xrightarrow{ d_\tau\log} A(n)[-1]$ and $ \ \sK_n\Omega^\infty_Q \supset (\sK_n\Omega^\infty_Q)_0$ to be the subcomplex $ \sK_n  \xrightarrow{d_\tau\log} A^n_{D_\tau}$, where $A^n_{D_\tau}$ means the subsheaf of $D_\tau$-closed sections. 
\end{defn}
\begin{lem} \label{lem3:4}
The $\tau$-connections on $(\xi_I^F)^s$ define a class $\xi^s(\nabla)\in \H^1(Q, (\sK_1\Omega^\infty_Q)_0)$ with the property that the image of $\xi^s(\nabla)$ in $H^1(Q, \sK_1)$ is $c_1(\xi^s)$. 
\end{lem}
\begin{proof}
 The cocycle of the class $\xi^s(\nabla)$ results from the claim \ref{claim3:1}.  Write $g_{\alpha \beta}^s$ for a $\sK_1$-valued 1-cocyle for $\xi^s$. Then the flat $\tau$-connection on $\xi^s$ is defined by local sections $\omega^s_\alpha$ in $q^*\Omega^1_X(\log D)$ which are $d_\tau$ flat for  $d_\tau: q^*\Omega^1_X(\log D)\to q^*\Omega^2_X(\log D)$. So the cocyle condition reads $d_\tau \log g_{\alpha \beta}^s=\delta (\omega^s)_{\alpha \beta}$ where $\delta$ is the Cech differential. The claim \ref{claim3:1} implies then that $\mu^*_I(\omega^s_{\alpha}) \in  (q_I^F)^*\Omega^1_{D_I}(\log (D_I\cap D^I))$, is $\tau$-flat and one has 
$d_\tau \log \mu_I^*(g_{\alpha \beta}^s)=\delta \mu^*_I(\omega^s)_{\alpha \beta}$. So the class 
$(\xi_I^F)^s$ is defined by the Cech cocyle $(g^s_{\alpha \beta},\mu_I^* \omega^s\oplus 0)$, with $\mu_I^* \omega^s \in B^1, 0\in C^1$. 

\end{proof}
We define a product 
\ga{3.5}{(\sK_m\Omega_{Q}^\infty)_0\times (\sK_n\Omega_{Q}^\infty)_0 \xrightarrow{\cup}
(\sK_{m+n}\Omega_{Q}^\infty)_0}
by using the formulae defined in \cite[Definition~2.1.1]{ADC}, that is 
\ga{3.6}{x\cup y=\begin{cases}
\{x,y\} & x\in \sK_m, y\in \sK_n\\
d_\tau \log x \wedge y \oplus  d_\tau \log x\wedge y & x\in \sK_m, y \in (B^n\oplus C^n)_{D_\tau}\\
0 & {\rm else}.
                \end{cases}
}
The product is well defined. 
\begin{defn} \label{defn3:5}
 We define $c_n(q^*(E,\nabla,\Gamma))\in \H^n(Q, \sK_n\Omega_{Q}^\infty))$ to be the image via the map  $\H^n(Q, (\sK_n\Omega_{Q}^\infty)_0)\to \H^n(Q, \sK_n\Omega_{Q}^\infty)$ of 
$$\sum_{s_1<s_2\ldots <s_n} \xi^{s_1}(\nabla)\cup \cdots \cup \xi^{s_n} (\nabla).$$

\end{defn}
\begin{defn} \label{defn3:6}
On $X$ we define the complex of sheaves $$A_X(n)=A_X^n\to A_X^{n+1}\to \ldots$$ with
\ml{}{A_X^i= B_X^i\oplus C_X^i\\  B_X^i=\oplus_I (i_I)_* \Omega^i_{D_I}(\log (D^I\cap D_I)), \ C^i_X=\oplus_{I\neq \emptyset}   (i_I)_*\Omega^{i-1}_X(\log D)|_{D_I},\notag } 
where $C_X^i=0$ for $i=n$.  The differentials $D_X$  are defined as follows: $(\oplus_I\beta_I, \oplus_I \gamma_I)$, where $\beta_I \in (i_I)_* \Omega^i_{D_I}(\log (D^I\cap D_I)), \gamma_I\in 
(i_I)_*\Omega^{i-1}_X(\log D)|_{D_I}$
 is sent to 
\ml{}{\oplus_I d \beta_I \in (i_I)_* \Omega^{i+1}_{D_I}(\log (D^I\cap D_I)),\\ \oplus_I d \gamma_I +(-1)^i (i_I^*\beta-\beta_I) \in (i_I)_*\Omega^{i}_X(\log D)|_{D_I},\notag} 
where the differentials $d_\tau$ are the $\tau$ differentials in the various differential graded algebras  $\Omega^\bullet_{D_I}(\log (D^I\cap D_I))$. 

\end{defn}
 One has an injective morphism of complexes 
\ga{3.7}{\iota:\Omega^{\ge n}_X\to A_X^{\ge n}}
sending $\alpha\in \Omega^i_X$ to $i_I^*\alpha \oplus 0$. 

\begin{prop} \label{prop3:7}
 The morphism $\iota$ is a quasi-isomorphism. Furthermore, one has 
$Rq_*A(n)=A_X(n)$. 
\end{prop}
\begin{proof}
 We start with the second assertion: since $\mu_I$ is a closed embedding, one has $R(\mu_I)_*=(\mu_I)_*$ on coherent sheaves. Thus by the commutativity of the diagram \eqref{3.2}, and the fact that $\sO$ on the flag varieties is relatively acyclic, one has $Rq_*(R\mu_I)_*(q_I^F)^* \sE=  (i_I)_* \sE$ for a locally free sheaf $\sE$ on $D_I$. This shows the second statement. We show the first assertion. We first show that the 0-th cohomology sheaf of $A_X(n)$ is $(\Omega^n_X)_d$. The condition  $D(\beta, \beta_I)=0$ means $d\beta=d\beta_I=0$ and $i_I^*\beta=\beta_I$. Thus $\beta\in \Omega^n_X$ and  $d\beta=0$. Assume now $i\ge n+1$. Then  modulo $DA^{i-1}(n)$, $((\beta, \beta_I), \gamma_I)$ is equivalent to 
$((\beta, \beta_I + (-1)^{i-1}d\gamma_I), 0)$. So we are back to the computation as in the case $i=n$ and 
${\rm Ker}(D)$ on $B^i\oplus 0$ is ${\rm Ker}(d)$ on $\Omega^i_X$. On the other hand, by the same reason,  $D(B^{i-1}\oplus C^{i-1})=D(B^{i-1}\oplus 0)$, and $D(B^{i-1}\oplus 0)\cap (B^i\oplus 0)=d(\Omega^i_X)$. This finishes the proof. 
\end{proof}
\begin{prop} \label{prop3:8}
 The map $q^*: AD^n(X)_\infty=\H^n(X, \sK_n\xrightarrow{d\log} \Omega^n_X\xrightarrow{d} \ldots \xrightarrow{d} \Omega^{{\rm dim}_X})\to \H^n(Q, \sK_n\Omega^\infty)$ is injective. The classes $c_n((q^*(E,\nabla,\Gamma))\in \H^n(Q, \sK_n\Omega^\infty)$ in definition \ref{defn3:5} are of the shape $q^*c_n((E,\nabla, \Gamma))$ for uniquely defined classes 
$c_n((E,\nabla, \Gamma))\in \H^n(X, \sK_n\xrightarrow{d\log} \Omega^n_X\xrightarrow{d} \ldots \to \Omega^{{\rm dim}_X})$.
\end{prop}
\begin{proof}
One has a commutative diagram of long exact sequences
\ga{3.8}{\xymatrix{ H^{n-1}(Q,\sK_n) \ar[r] & \H^{n-1}(A(n)) \ar[r] & \H^n(\sK_n\Omega^\infty_{Q}) \ar[r] & H^n(Q, \sK_n)\\
\ar[u]^{{\rm inj}} H^{n-1}(X,\sK_n) \ar[r] & \ar[u]^{=} \H^{n-1}(A(n)_X) \ar[r] & \ar[u] \H^n(\sK_n\Omega^\infty_{X}) \ar[r] & \ar[u]^{{\rm inj}} H^n(X, \sK_n)
}}
where $\sK_n\Omega^\infty_{X}=\sK_n\xrightarrow{d\log} \Omega^n_X\xrightarrow{d} \ldots \xrightarrow{d}\Omega_X^{{\rm dim}X}$. We write $H^{i}(Q, \sK_j)=H^i(X, \sK_j) \oplus {\rm rest}$, where the rest is divisible by the classes of powers of the $[\xi^s] \in H^1(Q, \sK_1)$, with coefficients in some $H^a(X, \sK_b)$.  But $[\xi^s]$ comes by lemma \ref{lem3:4} from a class 
$\xi^s(\nabla) \in \H^1(Q, (\sK_1\Omega^\infty_Q)_0)$. Consequently, the image of ${\rm rest}$ in 
$\H^{i}(A(n))$ dies.  We conclude that one has an exact sequence 
$0\to \H^n(\sK_n\Omega^\infty_{X}) \to \H^n(\sK_n\Omega^\infty_{Q}) \to 
\H^n(X, R^\bullet q_*\sK_n/q_*\sK_n)$.  By the standard splitting principle for Chow groups, one has $H^n(Q, \sK_n)/H^n(X, \sK_n)= \H^n(X, R^\bullet q_*\sK_n/q_*\sK_n)$, and  $$\sum_{s_1<s_2\ldots <s_n} c_1(\xi^{s_1})\cup \cdots \cup c_1(\xi^{s_n}) \in {\rm Im}( CH^n(X)\subset CH^n(Q)).$$  By lemma \ref{lem3:4}, $\xi^s(\nabla)
\in \H^1(Q, (\sK_1\Omega^\infty)_0)$ maps to $c_1(\xi^s)\in H^1(Q, \sK_1)$. Thus
we conclude that 
$c_n(q^*(E,\nabla, \Gamma))\in {\rm Im}( \H^n(\sK_n\Omega^\infty_{X})\subset \H^n(\sK_n\Omega^\infty_{Q})$. This finishes the proof.

\end{proof}
\begin{thm} \label{thm3:9}
Let $ X\supset U$ be a smooth (partial) compactification of a   variety $U$ defined over a characteristic 0 field, such that $D=\sum_j D_j= X\setminus U$ is a strict normal crossings divisor. Let $(E, \nabla)$ be a flat connection with logarithmic poles along  $D$ such that its residues $\Gamma_j$ along $D_j$ are all nilpotent.  Then the classes $c_n(( E,  \nabla)) \in AD^n( X,  D)$ lift to well defined classes
$c_n(( E,  \nabla, \Gamma)) \in AD^n( X)$. They are functorial: 
if  $f: Y\to X$  with $Y$ smooth, such that $f^{-1}(D)$ is a normal crossings divisor, \'etale over its image $\subset D$, then $f^*c_n((E,\nabla, \Gamma))=c_n(f^*(E,\nabla,\Gamma))$
in $AD^n(Y)$. If $D'\supset D$ is a normal crossings divisor and $\nabla'$ is the connection $\nabla$, but considered with logarithmic poles along $D'$, thus with trivial residues along the components of $D'\setminus D$, then $c_n((E,\nabla,\Gamma))=c_n((E,\nabla',\Gamma'))$. 
The classes $c_n((E,\nabla,\Gamma))$  satisfy the Whitney product formula. 
In addition,  $c_n(( E,  \nabla, \Gamma))$ lies  in the subgroup $AD^n_{\infty}( X)=\H^n( X, \sK_n\xrightarrow{d\log} \Omega^n_{\bar X}\xrightarrow{d} \Omega^{n+1}_{ X}\to \ldots \xrightarrow{d} \Omega^{{\rm dim}(X)}_{ X})\subset AD^n( X) $ of classes mapping to 0 in $H^{0}( X, \Omega^{ 2n}_X)$. The restriction to $AD^n_\infty(D_I)$ of $c_n((E,\nabla, \Gamma))$ is $c_n ((gr(F_I^\bullet),\nabla_I, \Gamma_I))$ where $(gr(F_I^\bullet),\nabla_I, \Gamma_I)$ is 
the canonical filtration  (see \ref{claim2:4} and \ref{defn2:5}). .

\end{thm}
\begin{proof}
The construction is the proposition \ref{prop3:8}. 
We discuss functoriality. If $f$ is as in the theorem, then the filtrations $F_I^\bullet$
for $(E,\nabla)$  restrict to the filtration for $f^*(E,\nabla)$. 
Whitney product formula is proven exactly as in \cite[2.17,2.18]{E} and \cite[Theorem~1.7]{E2}, even if this is more cumbersome,
as we have in addition to follow the whole tower of $F_I^\bullet$. 
 Finally, the last property
follows immediately from the definition of $\xi^s(\nabla)$ in lemma \ref{lem3:4}.
\end{proof}
\begin{thm} \label{thm3:9}Assume given $k\subset \C$ and 
 $\Gamma$ is nilpotent. Then the classes $\hat{c}_n((E,\nabla))\in H^{2n}((X\setminus D)_{{\rm an}}, \C/\Z(n))$ defined in \cite{E},  come from well defined classes $\hat{c}_n
((E,\nabla, \Gamma))\in H^{2n-1}(X_{{\rm an}}, \C/\Z(n))$.
Furthermore $\hat{c}_n((E,\nabla,\Gamma))$ fulfill the same functoriality, additivity, restriction and enlargement of $\nabla$ properties as $c_n((E,\nabla, \Gamma)) \in AD^n_\infty(X)$.
\end{thm}
\begin{proof}
We just have to use the regulator map  $AD^n(X)\to H^{2n-1}(X_{{\rm an}}, \C/\Z(n))$, which is an algebra homomorphism, and which  defined in \cite[Theorem~1.7]{E2}.
 Of course we can also follow  the same  construction directly in the analytic category. 
\end{proof}
\bibliographystyle{plain}
\renewcommand\refname{References}

\end{document}